\documentclass[10pt,reqno]{amsart}
\usepackage{amsaddr}
\usepackage{graphicx}
\usepackage{xcolor}
\newcommand{\new}[1]{#1}
\usepackage{amssymb}
\usepackage{amsfonts}
\usepackage{amsmath,amsthm}
\usepackage{float}
\usepackage{setspace}
\usepackage[english]{babel}
\usepackage[backref=page]{hyperref}
\newtheorem{theorem}{Theorem}[section]
\newtheorem{lemma}[theorem]{Lemma}
\newtheorem{definition}[theorem]{Definition}
\newtheorem{proposition}[theorem]{Proposition}
\newtheorem{corollary}[theorem]{Corollary}

\theoremstyle{remark}
\newtheorem{remark}[theorem]{Remark}
\numberwithin{equation}{section}
\newcommand{\R}{\mathbb R}

\newcommand{\Com}{\mathbb{C}}
\newcommand{\om}{\omega}
\newcommand{\pd}{\partial}
\newcommand{\la}{\lambda}
\newcommand{\ga}{\gamma}
\newcommand{\re}{\operatorname{Re}}
\newcommand{\ima}{\operatorname{Im}}
\newcommand{\supp}{\operatorname{supp}}

\newcommand{\dist}{\operatorname{dist}}
\newcommand{\Res}{\operatorname{Res}}

\newcommand{\ve}{\varepsilon}
\newcommand{\be}{\begin{equation}}
\newcommand{\ee}{\end{equation}}

\newcommand{\lan}{\langle}
\newcommand{\ran}{\rangle}
\newcommand{\Lc}{\mathcal{L}}
\newcommand{\Kc}{\mathcal{K}}
\newcommand{\F}{\mathcal{F}}

\newcommand{\sig}{\sigma}
\begin{document}
\title[KdV on the half-line at $H^{-3/4}$]{Local well-posedness for the KdV equation on the half-line at the critical regularity $H^{-\frac34}$}

\author{Márcio Cavalcante}
\address{Instituto de Matemática\\ Universidade Federal de Alagoas, Brazil}
\email{marcio.melo@im.ufal.br}

\thanks{AMS Subject Classifications: 35Q53, 35G31, 35A01.}

\maketitle

\begin{abstract}
We prove local well-posedness for the initial-boundary-value problem for the Korteweg--de Vries equation on the right half-line at the critical regularity $s=-3/4$, with initial and boundary data in $H^{-3/4}(\mathbb R_+)$ and $H^{1/12}(\mathbb R_+)$, respectively. This reaches the endpoint left open in the previous half-line results of Holmer (2006) and Bona, Sun, and Zhang (2006), who established local well-posedness for $s>-3/4$. The proof is based on a contraction argument in the Besov-type Bourgain space introduced by Kishimoto. The main difficulty is to construct a boundary forcing operator compatible with the endpoint $b=1/2$ and the required $\ell^1$-summability in the modulation variable. We construct such an operator using a Laplace-transform representation of the linear half-line problem and a carefully designed extension to $x<0$. The extension combines the required regularity across the boundary with prescribed vanishing spatial moments, yielding favorable low-frequency behavior of the associated kernel. A pointwise kernel estimate then provides the crucial dyadic summability needed for the $X^{-3/4,1/2,1}$ estimate. Together with a time-trace estimate for the Duhamel term at $b=1/2$, this yields the critical local well-posedness result. The ideas used here have prospects to be used in other nonlinear dispersive equations on the half-lines.
\end{abstract}

\section{Introduction}

Initial-boundary value problems for nonlinear dispersive equations on half-lines have attracted considerable attention in recent years, and become genuinely difficult at low Sobolev regularity.

In this paper we study the initial-boundary-value problem (IBVP) for the Korteweg--de Vries (KdV) equation posed on the right half-line,
\begin{equation}\label{IBVP}
	\begin{cases}
		\partial_tu+\partial_x(\partial_x^2 u+u^2)=0  ,& (x,t)\in \mathbb R^+\times(0,T),\\
		u(x,0)=u_0(x),                                   & x\in \mathbb R^+,\\
		u(0,t)=f(t),                                     & t\in(0,T),
	\end{cases}
\end{equation}
in the setting
\begin{equation}\label{setting_A}
(u_0,f)\in H^{-3/4}(\R^+)\times H^{\frac{1}{12}}(\R^+).
\end{equation}
The relation between the two regularities in \eqref{setting_A} is dictated by the localized smoothing effect of Kenig, Ponce and Vega \cite{KPV} for the Airy group $e^{-t\pd_x^3}$,
\begin{align*}
&\|\psi(t) e^{-t\partial_x^3}\phi(x)\|_{C\big(\mathbb{R}_x;\; H^{(s+1)/3}(\mathbb{R}_t)\big)}\lesssim \|\phi\|_{H^s(\mathbb{R})},
\end{align*}
$\psi$ a smooth time cutoff, which is sharp. As pointed out by Zabusky and Galvin \cite{Zabusky}, IBVPs of the form \eqref{IBVP} model waves generated by a wave maker in a channel, or waves approaching shallow water from deep water.

The IBVP \eqref{IBVP} has been extensively studied, following the works of Ton \cite{Ton}, Bona and Winther \cite{BW,BW2}, Faminskii \cite{Fa,Fa2}, Bona, Sun and Zhang \cite{BSZ1,BSZ2}, Colliander and Kenig \cite{CK}, Holmer \cite{Holmer}, Fokas \cite{Fokas1}, Fokas, Himonas and Mantzavinos \cite{Fokas2}, and Himonas and Yan \cite{HY}. Global well-posedness holds for $s\ge0$ in the setting \eqref{setting_A} (up to an arbitrarily small loss in the boundary datum when $s<1$), by Faminskii \cite{Fa2}, and the best local result is local well-posedness for $s>-\frac34$, obtained independently by Holmer \cite{Holmer} (via an analytic family of Duhamel boundary forcing operators and Bourgain spaces $X^{s,b}$ with $b<\frac12$) and by Bona, Sun and Zhang \cite{BSZ1,BSZ2} (via a Laplace transform representation of the linear solution). Fokas \cite{Fokas1,Fokas2}, and, through the unified transform method, Himonas and Yan \cite{HY} and Gallego and Kwak \cite{GK}; see also Compaan and Tzirakis \cite{CT} for the quartic gKdV. The endpoint $s=-\frac34$, which for the Cauchy problem on the line was settled by Kishimoto \cite{Kishimoto} and Guo \cite{Guo}, remained open on the half-line (see \cite{Cav} for a formulation of this question and of the strategy pursued here); it is the subject of the present paper.

For the Cauchy problem on the line,
\begin{equation}\label{pure}
\left\{\begin{array}{l}
\partial_{t} u+\partial_{x}^{3} u+\pd_x( u^2)=0, \quad x, t \in \mathbb{R},  \\
u(x, 0)=u_{0}(x),
\end{array}\right.
\end{equation}
the iteration method in Bourgain spaces gives local well-posedness for $s>-\frac34$ \cite{KPV96}, the bilinear estimate fails at $s=-\frac34$ for every $b$ \cite{NTT}, and the data-to-solution map fails to be uniformly continuous for $s<-\frac34$ \cite{CCT}. The endpoint $s=-\frac34$ was reached by the iteration method by Kishimoto \cite{Kishimoto} and, independently, by Guo \cite{Guo}, by replacing $X^{-\frac34,\frac12}$ with an $\ell^1$-Besov modification in high frequency and a further modification in low frequency; the corresponding bilinear estimate is Proposition \ref{prop:kishimoto} below. See also \cite{Killip} for the sharp result $s\ge-1$ by completely different methods.

\subsection{Main result}

Throughout, $s=-\frac34$ and $\ga=\frac{s+1}{3}=\frac1{12}$. Let $X$ be Kishimoto's space (Section \ref{sec:spaces}) and $X_T$ its restriction to $[-T,T]$.

\begin{definition}\label{def:solution}
Let $T>0$, $u_0\in H^{-3/4}(\R^+)$, $f\in H^{1/12}(\R^+)$. A function $u$ defined on $\R\times[-T,T]$ is a \emph{distributional solution} of \eqref{IBVP} on $[0,T]$ if
\begin{enumerate}
\item[(a)] (\emph{well-defined nonlinearity}) $u\in X_T$; by Proposition \ref{prop:kishimoto} below, $\pd_x(u^2)$ is then a well-defined distribution on $\R\times(-T,T)$;
\item[(b)] $\pd_t u+\pd_x^3u+\pd_x(u^2)=0$ in the sense of distributions on $(0,+\infty)\times(0,T)$;
\item[(c)] (\emph{space traces}) $u\in C([-T,T];H^{-3/4}(\R))$ and $u(\cdot,0)|_{\R^+}=u_0$ in $H^{-3/4}(\R^+)$;
\item[(d)] (\emph{time traces}) $u\in C(\R_x;H^{1/12}(0,T))$ and $u(0,\cdot)=f$ in $H^{1/12}(0,T)$.
\end{enumerate}
\end{definition}

At the endpoint, however, the Besov refinement creates a new difficulty for the half-line problem: the boundary forcing operator must be compatible with the $\ell^1$-summability in the modulation variable, and the existing boundary forcing constructions do not provide the required $\ell^1$-Besov estimate at $b=\frac12$. We overcome this obstruction by constructing an explicit boundary forcing operator whose extension to $x<0$ incorporates moment cancellations and yields the required $X^{-3/4,\frac12,1}$ estimate. We now state the main result of this work.

\begin{theorem}\label{thm:main}
Let $u_0\in H^{-3/4}(\R^+)$ and $f\in H^{1/12}(\R^+)$. There exist $T>0$, depending only on $\|u_0\|_{H^{-3/4}(\R^+)}+\|f\|_{H^{1/12}(\R^+)}$, and a distributional solution $u$ of \eqref{IBVP} on $[0,T]$ in the sense of Definition \ref{def:solution}. The map $(u_0,f)\mapsto u$ is Lipschitz (indeed analytic) from a neighbourhood of the data in $H^{-3/4}(\R^+)\times H^{1/12}(\R^+)$ into $X_T\cap C([0,T];H^{-3/4}(\R))\cap C(\R_x;H^{1/12}(0,T))$.
\end{theorem}

Theorem \ref{thm:main} is the half-line counterpart of Kishimoto's result \cite{Kishimoto} for \eqref{pure}, and the endpoint of Holmer's Theorem 1.3(a) \cite{Holmer}. As in \cite{Holmer}, uniqueness is understood as uniqueness of the fixed point of the integral equation \eqref{eq:scheme} below; unconditional uniqueness in $X_T$ is not addressed (for $s>-\frac34$ see \cite{BSZ2}).

\subsection{Strategy and the obstruction at $b=\frac12$}\label{sec:strategy}

The proof follows the scheme: seek $u$ on the whole line as a fixed point of
\begin{equation}\label{eq:scheme}
u=\psi(t)e^{-t\pd_x^3}\tilde u_0-\psi(t)\Kc\pd_x(u^2)+\Lc h,
\end{equation}
where $\tilde u_0$ extends $u_0$, $\Kc$ is the Duhamel operator, $\Lc$ is a boundary forcing operator producing solutions of the linear equation on $x>0$ with zero initial data and prescribed trace $h$ at $x=0$, and $h$ is chosen so that $u(0,t)=f(t)$. The fixed point space is Kishimoto's $X_T\cap C_tH^{-3/4}$, and the nonlinear term is controlled by Kishimoto's bilinear estimate. Two linear estimates are therefore needed beyond \cite{Kishimoto}: $\Lc$ must be bounded from $H^{1/12}$ into $X\cap C_tH^{-3/4}$, and the trace $\Kc\pd_x(u^2)|_{x=0}$ must be controlled in $H^{1/12}$ by the norm in which Kishimoto measures the nonlinearity.

The second estimate is Lemma \ref{lem:trace}; it is the $b=\frac12$, $\ell^1$-Besov version of Holmer's Lemma 5.6(b). The first is the real issue. The boundary forcing operators of Colliander--Kenig \cite{CK} and Holmer \cite{Holmer} are Duhamel integrals of a source $\delta_0(x)$ (or $x_-^{\la-1}/\Gamma(\la)$) placed at the boundary; their space-time Fourier transforms contain the factor $(\tau-\xi^3)^{-1}$ with a coefficient that does \emph{not} vanish on the real characteristic curve $\tau=\xi^3$. Such a function has, at fixed $\tau$, mass $\sim 2^{-k}$ on each dyadic shell $\{\lan\tau-\xi^3\ran\sim 2^k\}$, hence is bounded in $X^{s,b}$ only for $b<\frac12$ (this is Holmer's Lemma 5.8(d) and Remark 2.3 there), while Kishimoto's space requires $b=\frac12$ and summability in $k$. This is the reason why the direct combination of \cite{Holmer} and \cite{Kishimoto} does not close. The extensions $BI^{m1},BI^{m2}$ of \cite{BSZ2} do reach $b=\frac12$ in $X^{s,b}$; they are defined through several Fourier cutoffs and a corrector chosen by an implicit equation ((2.22) there), and the $\ell^1_k$ summability, which requires a pointwise understanding of the symbol on each dyadic shell around the characteristic curve, is not addressed there. This is not a technicality: at $s=-\frac34$ the bilinear estimate fails in $X^{-\frac34,\frac12}$ whatever the low-frequency modification \cite{NTT,Kishimoto2}, and $X^{s,\frac12}$ does not embed in $C_tH^s$, so the $\ell^1$-Besov refinement of \cite{Kishimoto} is forced. The same applies to the cutoff extensions of \cite{ET,HY,CT,GK}, which also reach $b=\frac12$ in $X^{s,b}$; a detailed comparison with all these constructions is given in Remark \ref{rem:holmer}.

The way out is to notice that, formally, the trace at $x=0$ of a Duhamel solution $\Kc F$ with $\supp F\subset\{x\le0\}$ is the integral $\int\widehat F(\xi,\tau)(\tau-\xi^3)^{-1}d\xi$, which can be nonzero even if $\widehat F$ vanishes on the real root $\xi=\tau^{1/3}$: it then picks up the residue at the \emph{complex} root $\xi=\om\,\tau^{1/3}$, $\om=e^{2\pi i/3}$, the decaying mode of the linear equation. We therefore define $\Lc h$ (Section \ref{sec:derivation}) as the exact solution $\frac1{2\pi}\int e^{it\tau}\widehat h(\tau)e^{ix\om(\tau-i0)^{1/3}}d\tau$ of the linear problem on $x>0$, the Laplace-transform formula of \cite{BSZ1}, continued at $x=0$ to an explicit decaying profile on $x<0$ with a prescribed number of derivatives. On the Fourier side this is a symbol $\widehat h(\tau)K(\xi,\tau)$ whose only pole in the closed upper half-plane, at fixed $\tau$, is a simple pole at $\xi=\om\,\tau^{1/3}$ (Definition \ref{def:L}). Hence $(\tau-\xi^3)K$ is analytic in the upper half-plane, so the forcing is supported in $x\le0$; the residue gives the trace; analyticity in $\tau$ in the lower half-plane gives causality; and, since $K$ is bounded near the real root, $\Lc h\in X^{s,\frac12,1}$ with norm $\sim\|h\|_{H^{(s+1)/3}}$.

\section{Notation and function spaces}\label{sec:spaces}

\subsection{Fourier transforms and complex powers}
We use the conventions
$$\widehat f(\xi)=\int_\R e^{-ix\xi}f(x)\,dx,\qquad \widehat u(\xi,\tau)=\iint_{\R^2}e^{-i(x\xi+t\tau)}u(x,t)\,dx\,dt,$$
and $\F^{-1}$ for the inverse transforms (with the factors $(2\pi)^{-1}$, $(2\pi)^{-2}$). We write $\lan\cdot\ran=(1+|\cdot|^2)^{1/2}$ and $\sig=\tau-\xi^3$. Constants $C$ may change from line to line; $A\lesssim B$ means $A\le CB$.

For $z\in\Com$ with $\ima z\le0$, $z\ne0$, and $\mu\in\R$, set $z^\mu:=|z|^\mu e^{i\mu\arg z}$ with $\arg z\in[-\pi,0]$. The function $z\mapsto z^\mu$ is holomorphic in $\{\ima z<0\}$ and continuous on $\{\ima z\le0\}\setminus\{0\}$; its restriction to $\R\setminus\{0\}$ is denoted $(\tau-i0)^\mu$:
\begin{equation}\label{eq:power}
(\tau-i0)^\mu=|\tau|^\mu\ \ (\tau>0),\qquad (\tau-i0)^\mu=|\tau|^\mu e^{-i\pi\mu}\ \ (\tau<0).
\end{equation}
This is the convention of \cite[(2.2)]{Holmer} (e.g.\ $\widehat{\mathbf 1_{t>0}}(\tau)=\int_0^\infty e^{-it\tau}dt=-i(\tau-i0)^{-1}$); the identity is not used below and only serves to fix the branch. We shall use
\begin{equation}\label{eq:rho}
\rho(\tau):=(\tau-i0)^{1/3},\qquad \rho(\tau)^3=\tau,\qquad \rho(\tau)=\begin{cases}|\tau|^{1/3}, & \tau>0,\\ |\tau|^{1/3}e^{-i\pi/3}, & \tau<0.\end{cases}
\end{equation}

\subsection{Sobolev spaces on the half-line}
 $H^s(\R)$ is the usual Sobolev space. For $s\ge0$, $\phi\in H^s(\R^+)$ if $\phi=\tilde\phi|_{\R^+}$ for some $\tilde\phi\in H^s(\R)$, with $\|\phi\|_{H^s(\R^+)}=\inf\|\tilde\phi\|_{H^s(\R)}$; $H^s_0(\R^+)$ is the space of $\phi\in H^s(\R)$ supported in $[0,+\infty)$; for $s<0$, $H^s(\R^+)$ and $H^s_0(\R^+)$ are defined by duality. Every $\phi\in H^s(\R^+)$ has an extension $\tilde\phi\in H^s(\R)$ with $\|\tilde\phi\|_{H^s(\R)}\le 2\|\phi\|_{H^s(\R^+)}$. $C_0^\infty(\R^+)$ denotes smooth functions on $\R$ supported in $[0,+\infty)$ and $C^\infty_{0,c}(\R^+)$ those with compact support; $C^\infty_{0,c}(\R^+)$ is dense in $H^s_0(\R^+)$ for all $s$. We shall need two facts.

\begin{lemma}[{\cite[Lemma 3.5]{JK}}, {\cite[Lemma 4.2]{Holmer}}]\label{lem:JK}
If $-\frac12<\alpha<\frac12$, then $\|\chi_{(0,+\infty)}g\|_{H^\alpha(\R)}\le c(\alpha)\,\|g\|_{H^\alpha(\R)}$ for all $g\in H^\alpha(\R)$.
\end{lemma}

\begin{lemma}\label{lem:cutoff}
Let $\alpha\ge0$ and $\theta\in C_c^\infty(\R)$. Then $\|\theta g\|_{H^\alpha(\R)}\le C\|\theta\|_{H^{\alpha+1}(\R)}\|g\|_{H^\alpha(\R)}$.
\end{lemma}
\begin{proof}
$\widehat{\theta g}=(2\pi)^{-1}\widehat\theta*\widehat g$ and $\lan\tau\ran^\alpha\le\lan\tau-\tau'\ran^\alpha\lan\tau'\ran^\alpha$, so $\|\theta g\|_{H^\alpha}\le(2\pi)^{-1}\|\lan\tau\ran^\alpha\widehat\theta\|_{L^1}\|g\|_{H^\alpha}$, and $\|\lan\tau\ran^\alpha\widehat\theta\|_{L^1}\le\|\lan\tau\ran^{-1}\|_{L^2}\|\theta\|_{H^{\alpha+1}}$.
\end{proof}

\subsection{Bourgain and Besov--Bourgain spaces}
For $s,b\in\R$, $X^{s,b}$ is the completion of $\mathcal S(\R^2)$ under
$$\|u\|_{X^{s,b}}=\|\lan\xi\ran^s\lan\tau-\xi^3\ran^b\widehat u\|_{L^2_{\tau,\xi}}.$$
Following \cite{Kishimoto}, define the dyadic sets, for integers $j,k\ge0$,
\begin{equation}\label{besov}
A_j := \{ (\tau, \xi) : 2^j \leq \langle \xi \rangle < 2^{j+1} \}, \qquad
B_k := \{ (\tau, \xi) : 2^k \leq \langle \tau - \xi^3 \rangle < 2^{k+1} \},
\end{equation}
and the norm
\begin{equation}\label{eq:Xsb1}
\|u\|_{X^{s,b,1}} := \Big( \sum_{j\ge0} \Big( \sum_{k\ge0} \|\langle \xi \rangle^s \langle \tau - \xi^3 \rangle^b \hat{u} \|_{L^2_{\tau,\xi}(A_j\cap B_k)} \Big)^2 \Big)^{1/2}.
\end{equation}
Clearly $\|u\|_{X^{s,b}}\le\|u\|_{X^{s,b,1}}$, and $\|u\|_{X^{s,b,1}}\le C_{b'}\|u\|_{X^{s,b'}}$ for $b'>b$. Note that for $j\ge1$ and $(\tau,\xi)\in A_j$ one has $2^{j-1}\le|\xi|<2^{j+1}$, while $|\xi|<\sqrt3$ on $A_0$.

\begin{lemma}[{\cite[Section 2]{Kishimoto}}]\label{lem:embedding}
$\|u\|_{L^\infty_tH^s_x}\le(2\pi)^{-1}\|\lan\xi\ran^s\widehat u\|_{L^2_\xi L^1_\tau}\le C\|u\|_{X^{s,\frac12,1}}$, and $X^{s,\frac12,1}\hookrightarrow C(\R_t;H^s(\R_x))$.
\end{lemma}
\begin{proof}
The first inequality is Minkowski's. For the second, by Cauchy--Schwarz in $\tau$ on each $B_k$, $\|\widehat u(\xi,\cdot)\|_{L^1_\tau(B_k)}\le 2^{(k+1)/2}\|\widehat u(\xi,\cdot)\|_{L^2_\tau(B_k)}$; summing in $k$ and taking the $L^2_\xi(A_j)$ norm with Minkowski's inequality gives $$\|\lan\xi\ran^s\widehat u\|_{L^2_\xi(A_j)L^1_\tau}\le\sqrt2\sum_k 2^{k/2}\|\lan\xi\ran^s\widehat u\|_{L^2(A_j\cap B_k)};$$ the $\ell^2_j$ sum gives the claim. Continuity in $t$ follows by density of $\mathcal S$.
\end{proof}

Let
\begin{equation}\label{eq:D}
D := \{ (\tau, \xi) \in \mathbb{R}^2 : |\xi| \leq 1,\ |\tau| \geq |\xi|^{-3} \}.
\end{equation}
Kishimoto's space $X$ and auxiliary space $Y$ are defined by the norms
\begin{equation}\label{eq:X}
\|u\|_X := \| \mathcal{F}^{-1} [\mathbf{1}_{\mathbb{R}^2 \setminus D} \hat{u}] \|_{X^{-\frac{3}{4}, \frac{1}{2}, 1}} + \| \mathcal{F}^{-1} [\mathbf{1}_D \hat{u}] \|_{X^{-\frac{3}{4}, \frac{1}{2}}},\qquad
\|u\|_Y:=\|\lan\xi\ran^{-\frac34}\widehat u\|_{L^2_\xi L^1_\tau}.
\end{equation}
Then $\|u\|_X\le 2\|u\|_{X^{-\frac34,\frac12,1}}$ (each of the two terms is bounded by $\|u\|_{X^{-\frac34,\frac12,1}}$; the constant $2$ cannot be replaced by $1$) and $\|u\|_{X^{-\frac34,\frac12}}\le\|u\|_X$. For $T>0$, $X_T$ is the space of restrictions to $\R\times[-T,T]$ of elements of $X$, with the norm
\begin{equation}\label{eq:XT}
\|u\|_{X_T} := \inf \{ \|v\|_X : v \in X,\ u(t) = v(t) \text{ for } -T \leq t \leq T \}.
\end{equation}
Finally, for a distribution $G$ we write
\begin{equation}\label{eq:Nnorm}
\|G\|_{\mathcal N}:=\big\|\F^{-1}\big[\lan\tau-\xi^3\ran^{-1}\widehat G\big]\big\|_X+\big\|\F^{-1}\big[\lan\tau-\xi^3\ran^{-1}\widehat G\big]\big\|_Y ;
\end{equation}
this is the norm in which \cite{Kishimoto} measures the nonlinearity.

\subsection{Kishimoto's estimates}
The Duhamel operator is
$$\Kc F(x,t)=\int_0^t e^{-(t-t')\pd_x^3}F(x,t')\,dt',\qquad (\pd_t+\pd_x^3)\Kc F=F,\quad \Kc F(x,0)=0 .$$
The following two results are \cite[Proposition 1.1 and Lemma 4.1]{Kishimoto}.

\begin{proposition}[Bilinear estimate, \cite{Kishimoto}]\label{prop:kishimoto}
There is $C>0$ such that for all $u,v\in X$,
$$\|\pd_x(uv)\|_{\mathcal N}\le C\|u\|_X\|v\|_X .$$
In particular the bilinear map $(u,v)\mapsto\pd_x(uv)$, defined on $\mathcal S(\R^2)$, extends continuously to $X\times X$ with values in the space of tempered distributions $G$ with $\|G\|_{\mathcal N}<\infty$.
\end{proposition}

\begin{lemma}[Linear estimates, \cite{Kishimoto}]\label{lem:kishimoto-linear}
For $0<T\le1$, with constants independent of $T$,
\begin{align}
\|e^{-t\pd_x^3}u_0\|_{X_T}+\sup_{|t|\le T}\|e^{-t\pd_x^3}u_0\|_{H^{-3/4}}&\lesssim\|u_0\|_{H^{-3/4}},\label{eq:K41}\\
\|\Kc G\|_{X_T}+\sup_{|t|\le T}\|\Kc G(t)\|_{H^{-3/4}}&\lesssim\|G\|_{\mathcal N}.\label{eq:K42}
\end{align}
\end{lemma}

\subsection{Traces of the free group}
Let $\psi\in C_c^\infty(\R)$ with $\psi=1$ on $[-1,1]$ and $\supp\psi\subset[-2,2]$.

\begin{lemma}[\cite{KPV}, \cite{Holmer}]\label{lem:KPV}
Let $s\in\R$, $\ga=\frac{s+1}3\in[0,\frac12)$, and $\theta\in C_c^\infty(\R)$ with $\supp\theta\subset[-2,2]$. Then
$$\sup_{x\in\R}\|\theta(t)e^{-t\pd_x^3}\phi(x,\cdot)\|_{H^\ga(\R_t)}\le C\|\theta\|_{H^1(\R)}\|\phi\|_{H^s(\R)},$$
and $x\mapsto\theta(t)e^{-t\pd_x^3}\phi(x,\cdot)$ is continuous into $H^\ga(\R_t)$.
\end{lemma}

\section{The time trace of the Duhamel operator at $b=\frac12$}\label{sec:trace}

The following lemma is the Besov-type, $b=\frac12$ replacement of \cite[Lemma 5.6(b)]{Holmer}. Its right-hand side is precisely the quantity controlled by Proposition \ref{prop:kishimoto}. We state it for general $s$ since the proof is the same.

\begin{lemma}\label{lem:trace}
Let $-1<s<\frac12$, $\ga=\frac{s+1}3$, and let $\|\cdot\|_{\mathcal N_s}$ be defined as in \eqref{eq:Nnorm} with $-\frac34$ replaced by $s$ in \eqref{eq:X}. Then for every $G$ with $\|G\|_{\mathcal N_s}<\infty$,
$$\sup_{x\in\R}\|\psi(t)\Kc G(x,\cdot)\|_{H^\ga(\R_t)}\le C\|G\|_{\mathcal N_s},$$
and $x\mapsto\psi(t)\Kc G(x,\cdot)$ is continuous into $H^\ga(\R_t)$.
\end{lemma}

\begin{proof}
Set $W:=\F^{-1}[\lan\sig\ran^{-1}\widehat G]$, so that $\|G\|_{\mathcal N_s}=\|W\|_{X}+\|W\|_{Y}$ (with $s$ in place of $-\frac34$). By density it suffices to prove the bound for $G\in\mathcal S(\R^2)$; the continuity in $x$ then follows from the uniform bound and the explicit formulas below. Let $\psi_0\in C_c^\infty(\R)$ with $\psi_0=1$ on $[-1,1]$ and $\supp\psi_0\subset[-2,2]$, applied to the variable $\sig$. Taking the Fourier transform in $x$,
$$\mathcal{F}_x {\Kc G}(\xi,t)=\frac{1}{2\pi}\, e^{it\xi^3}\int_\R\frac{e^{it\sig}-1}{i\sig}\,\widehat G(\xi,\tau)\,d\tau ,$$
and we decompose $\Kc G=I+II+III$ with
\begin{align*}
I&=\frac{1}{(2\pi)^2}\iint e^{ix\xi+it\xi^3}\frac{e^{it\sig}-1}{i\sig}\psi_0(\sig)\widehat G\,d\tau d\xi
=\sum_{k\ge1}\frac{i^{k-1}t^k}{k!}\,e^{-t\pd_x^3}\phi_k,\\ 
&\quad \quad \widehat{\phi_k}(\xi)=\frac1{2\pi}\int\sig^{k-1}\psi_0(\sig)\widehat G\,d\tau,\\
&II=\frac{1}{(2\pi)^2}\iint e^{ix\xi+it\tau}\frac{1-\psi_0(\sig)}{i\sig}\widehat G\,d\tau d\xi,\\
&III=-\,e^{-t\pd_x^3}\phi_3,\qquad \widehat{\phi_3}(\xi)=\frac1{2\pi}\int\frac{1-\psi_0(\sig)}{i\sig}\widehat G\,d\tau .
\end{align*}

\emph{Term $I$.} By Lemma \ref{lem:KPV} with $\theta_k=t^k\psi$, $\|\theta_k\|_{H^1}\le C^k$, we get $$\|\psi I(x,\cdot)\|_{H^\ga}\le\sum_k\frac{C^k}{k!}\|\phi_k\|_{H^s},$$ and since $|\sig^{k-1}\psi_0(\sig)|\le 2^{k-1}\lan\sig\ran^{-1}\cdot\sqrt5$,
$$\|\phi_k\|_{H^s}\le C2^{k}\big\|\lan\xi\ran^s\lan\sig\ran^{-1}\widehat G\big\|_{L^2_\xi L^1_\tau}=C2^k\|W\|_Y .$$

\emph{Term $III$.} Since $(1-\psi_0(\sig))/|\sig|\le\sqrt2\lan\sig\ran^{-1}$, $\|\phi_3\|_{H^s}\le C\|W\|_Y$, and Lemma \ref{lem:KPV} gives $\|\psi\, III(x,\cdot)\|_{H^\ga}\le C\|W\|_Y$.

\emph{Term $II$.} By Lemma \ref{lem:cutoff} it suffices to bound $\|II(x,\cdot)\|_{H^\ga}$. The time Fourier transform of $II(x,\cdot)$ is $(2\pi)^{-1}F_x(\tau)$, where
$$F_x(\tau):=\int_\R e^{ix\xi}\frac{1-\psi_0(\sig)}{i\sig}\widehat G(\xi,\tau)\,d\xi .$$
All bounds below are in terms of $|\widehat G|$, hence uniform in $x$; we drop $e^{ix\xi}$. Set
$$D_\tau:=\int_\R|\widehat G(\xi,\tau)|^2\lan\xi\ran^{2s}\lan\sig\ran^{-1}d\xi,\qquad \int_\R D_\tau\,d\tau\le 2\|W\|^2_{X^{s,\frac12}}\le2\|W\|_X^2 .$$

\emph{(i) $|\tau|\le8$.} By Cauchy--Schwarz, $|F_x(\tau)|^2\le D_\tau\int\lan\xi\ran^{-2s}\lan\sig\ran^{-1}(1-\psi_0(\sig))^2d\xi$. For $|\xi|\le4$ the last integrand is bounded, and for $|\xi|\ge4$ one has $|\sig|\ge|\xi|^3-8\ge\frac12|\xi|^3$, so the integral is $\le C\int_{|\xi|\ge4}|\xi|^{-2s-3}d\xi+C\le C$ because $s>-1$. Hence $\int_{|\tau|\le8}\lan\tau\ran^{2\ga}|F_x|^2d\tau\le C\|W\|_X^2$.

\emph{(ii) $|\tau|\ge8$.} Let $r=r(\tau)$ be the real cube root of $\tau$, so $|r|\ge2$, and split $\R_\xi=R_1\cup R_2\cup R_3$,
$$R_1=\{|\xi-r|\le\tfrac12|r|\},\qquad R_3=\{|\xi|\ge\tfrac32|r|\},\qquad R_2=\R\setminus(R_1\cup R_3).$$
Note that $\xi\in R_1$ means that $\xi$ has the sign of $r$ and $\frac12|r|\le|\xi|\le\frac32|r|$. Hence a point of $R_2$ either satisfies $|\xi|<\frac12|r|$, whence $|\xi^3|<\frac18|\tau|$ and $|\sig|=|\tau-\xi^3|\ge\frac78|\tau|$, or has the sign opposite to $r$ (otherwise, having $\frac12|r|\le|\xi|<\frac32|r|$, it would lie in $R_1$), whence $\tau$ and $\xi^3$ have opposite signs and $|\sig|=|\tau|+|\xi|^3\ge|\tau|$. In both cases $|\sig|\ge\frac78|\tau|$ on $R_2$. On $R_3$, $|\sig|\ge|\xi|^3-|\tau|\ge\frac{19}{27}|\xi|^3$. On $R_1$, $\frac12|r|\le|\xi|\le\frac32|r|$, hence $\lan\xi\ran\sim|r|=|\tau|^{1/3}$ and $|\pd_\xi\sig|=3\xi^2\ge\frac34r^2$.

On $R_2$, by Cauchy--Schwarz and $s<\frac12$,
$$|F_{x,R_2}(\tau)|^2\le D_\tau\,\frac{C}{|\tau|}\int_{|\xi|\le\frac32|r|}\lan\xi\ran^{-2s}d\xi\le C D_\tau|\tau|^{-1}|\tau|^{\frac{1-2s}3}=CD_\tau\lan\tau\ran^{-2\ga},$$
since $-1+\frac{1-2s}3=-\frac{2s+2}3=-2\ga$. On $R_3$, by Cauchy--Schwarz and $s>-1$,
$$|F_{x,R_3}(\tau)|^2\le D_\tau\, C\int_{|\xi|\ge\frac32|r|}|\xi|^{-2s-3}d\xi\le CD_\tau|r|^{-2s-2}=CD_\tau\lan\tau\ran^{-2\ga}.$$
Integrating in $\tau$, the contributions of $R_2$ and $R_3$ to $\|\lan\tau\ran^\ga F_x\|_{L^2}^2$ are $\le C\|W\|_X^2$.

On $R_1$ we use the Besov structure. For $k\ge0$ let $E_k(\tau)=\{\xi\in R_1:2^k\le\lan\sig\ran<2^{k+1}\}$, and let $|E_k(\tau)|$ denote its Lebesgue measure in $\xi$. We claim that
\begin{equation}\label{eq:Ek}
|E_k(\tau)|\le \tfrac{16}{3}\,2^{k}|\tau|^{-2/3}.
\end{equation}
Indeed, the condition $2^k\le\lan\sig\ran<2^{k+1}$ means $$\sig\in J_k:=\{\sig\in\R: 2^{2k}-1\le\sig^2<2^{2k+2}-1\},$$ which is the union of two intervals (one in $\sig>0$, one in $\sig<0$) each of length at most $2^{k+1}$, so $|J_k|\le2^{k+2}$. On $R_1$ the map $\xi\mapsto\sig(\xi)=\tau-\xi^3$ is strictly monotone ($\xi$ has the sign of $r$ and $|\xi|\ge\frac12|r|$), with $|\sig'(\xi)|=3\xi^2\ge\frac34r^2=\frac34|\tau|^{2/3}$. Hence, changing variables,
$$
|E_k(\tau)|=\int_{E_k(\tau)}d\xi=\int_{\sig(E_k(\tau))}\frac{d\sig}{|\sig'(\xi(\sig))|}\le\frac{|J_k|}{\frac34|\tau|^{2/3}}\le\frac{2^{k+2}}{\frac34|\tau|^{2/3}},
$$
which is \eqref{eq:Ek}. (Geometrically, $E_k(\tau)$ consists of two intervals on either side of $r$, at distance $\approx2^k/(3r^2)$ from $r$ and of width $\approx2^k/(3r^2)$: the shell $B_k$ cut at height $\tau$. It is empty when $2^k\gg|\tau|$, since $|\sig|\le|\tau|+|\xi|^3\le\frac{35}{8}|\tau|$ on $R_1$.) Moreover $\lan\xi\ran^{-2s}\le C|\tau|^{-2s/3}$ on $R_1$, because $\frac12|r|\le|\xi|\le\frac32|r|$ and $|r|\ge2$. Hence, by Cauchy--Schwarz and \eqref{eq:Ek},
\begin{equation}
\begin{split}
\Big|\int_{E_k(\tau)}\frac{1-\psi_0(\sig)}{i\sig}\widehat G\,d\xi\Big|&\le 2^{1-k}\|\lan\xi\ran^s\widehat G(\cdot,\tau)\|_{L^2(E_k(\tau))}\Big(\int_{E_k(\tau)}\lan\xi\ran^{-2s}d\xi\Big)^{1/2}
\\&\le C2^{-k/2}|\tau|^{-\frac{s+1}3}\|\lan\xi\ran^s\widehat G(\cdot,\tau)\|_{L^2(E_k(\tau))}.
\end{split}
\end{equation}
Let $j(\tau)$ be defined by $2^{j(\tau)}\le\lan r\ran<2^{j(\tau)+1}$; then $$E_k(\tau)\times\{\tau\}\subset\bigcup_{|j-j(\tau)|\le2}(A_j\cap B_k),$$ and since $|\xi|\ge1$ on $R_1$, these sets lie outside $D$. Therefore
$$\lan\tau\ran^\ga|F_{x,R_1}(\tau)|\le C\sum_{|j-j(\tau)|\le2}\ \sum_{k\ge0}2^{-k/2}\|\lan\xi\ran^s\widehat G(\cdot,\tau)\|_{L^2_\xi(A_j\cap B_k)} .$$
Fix $j_0\ge0$. Taking the $L^2_\tau$ norm over $\{\tau:|\tau|\ge8,\ j(\tau)=j_0\}$ and using Minkowski's inequality,
$$\big\|\lan\tau\ran^\ga F_{x,R_1}\big\|_{L^2(\{j(\tau)=j_0\})}\le C\sum_{|j-j_0|\le2}\sum_{k\ge0}2^{-k/2}\|\lan\xi\ran^s\widehat G\|_{L^2_{\tau,\xi}(A_j\cap B_k)}\le C\sum_{|j-j_0|\le2}a_j,$$
where $a_j:=\sum_k2^{k/2}\|\lan\xi\ran^s\mathbf 1_{\R^2\setminus D}\widehat W\|_{L^2(A_j\cap B_k)}$; here we used $\lan\sig\ran\le2^{k+1}$ on $B_k$, so that $2^{-k/2}|\widehat G|\le 2\cdot 2^{k/2}\lan\sig\ran^{-1}|\widehat G|=2\cdot2^{k/2}|\widehat W|$. Summing the squares over $j_0$,
$$\big\|\lan\tau\ran^\ga F_{x,R_1}\big\|_{L^2(|\tau|\ge8)}^2\le C\sum_{j_0}\Big(\sum_{|j-j_0|\le2}a_j\Big)^2\le C\sum_ja_j^2\le C\|W\|_X^2 .$$
Collecting (i) and (ii), $\|II(x,\cdot)\|_{H^\ga}\le C\|W\|_X$ uniformly in $x$, which completes the proof.
\end{proof}

\begin{remark}
The only place where the $\ell^1_k$ structure is used is the region $R_1$; there the Cauchy--Schwarz argument of \cite{Holmer} produces $\int_{|\sig|\lesssim|\tau|}\lan\sig\ran^{2b-2}d\sig$, which is finite only for $b<\frac12$ and diverges logarithmically at $b=\frac12$. The decisive point is that, for fixed $\tau$, the region $R_1$ lives in a single dyadic band $A_{j(\tau)}$ (up to two neighbours), so that the bound is summed in the order $\ell^2_j\ell^1_k$ of \eqref{eq:Xsb1} and not in the larger $\ell^1_k\ell^2_j$ norm. For the derivative traces $\pd_x^j\Kc G$, $j=1,2$, the same argument requires $s>j-1$; on the right half-line only $j=0$ is needed. \new{Note also that for $b<\frac12$ one has $X^{s,-b}\hookrightarrow\mathcal N_s$, since $\sum_k2^{-k/2}a_k\le C_b(\sum_k2^{-2bk}a_k^2)^{1/2}$ and $\|\lan\xi\ran^s\lan\sig\ran^{-1}\widehat G\|_{L^2_\xi L^1_\tau}\le\|\lan\sig\ran^{b-1}\|_{L^2_\tau}\|G\|_{X^{s,-b}}$; hence Lemma \ref{lem:trace} contains \cite[Lemma 5.6(b)]{Holmer} in the range $-1<s<\frac12$. Roughly, the Duhamel estimates at $b=\frac12$ (Lemma \ref{lem:kishimoto-linear} and Lemma \ref{lem:trace}) ask less of the nonlinearity and give more on the solution than their $b<\frac12$ counterparts in \cite{Holmer}.}
\end{remark}

\section{A boundary forcing operator at the endpoint $b=\frac12$}\label{sec:L}

\subsection{Derivation of the operator}\label{sec:derivation}
Before giving the definition, we explain where the symbol comes from. Throughout this subsection $\om=e^{2\pi i/3}$ and $\rho(\tau)=(\tau-i0)^{1/3}$ as in \eqref{eq:rho}; recall $\rho(\tau)^3=\tau$ and
\begin{equation}\label{eq:rho-props}
\rho(\tau)=|\tau|^{1/3}e^{i\theta_\tau},\quad \theta_\tau\in\{0,-\tfrac\pi3\},\qquad \ima\big(\om\rho(\tau)\big)=\tfrac{\sqrt3}2|\tau|^{1/3},\qquad \re\rho(\tau)\ge\tfrac12|\tau|^{1/3} .
\end{equation}

 Let $h\in C^\infty_{0,c}(\R^+)$ and let $u$ be the solution of the linear IBVP
\begin{equation}\label{eq:linIBVP}
\begin{cases}
\pd_tu+\pd_x^3u=0, & (x,t)\in\R^+\times\R^+,\\
u(x,0)=0, & x\in\R^+,\\
u(0,t)=h(t), & t\in\R^+,
\end{cases}
\end{equation}
which for such $h$ exists, is unique among smooth solutions decaying as $x\to+\infty$ (see the uniqueness computation \cite[(1.3)]{Holmer}), and vanishes for $t\le0$, since $h$ does. (The present subsection is a derivation; all properties of the resulting operator are proved independently in Proposition \ref{prop:L}.) We call a function of $t$ \emph{causal} if it vanishes for $t<0$; by the Paley--Wiener theorem (see e.g.\ \cite[Ch.~19]{Rudin}, \cite[Thm.~7.4.3]{Hormander}) this is equivalent to its Fourier transform being the boundary value of a function holomorphic in the lower half-plane with polynomial growth. Because $u$ is causal, its partial Fourier transform in time
$$\F_tu(x,\tau):=\int_\R e^{-it\tau}u(x,t)\,dt=\int_0^\infty e^{-it\tau}u(x,t)\,dt$$
is well defined and holomorphic in $\tau\in\Pi_-:=\{\ima\tau<0\}$, with boundary values on $\R_\tau$. For causal $u$ this is the Laplace transform $\int_0^\infty e^{-st}u\,dt$ evaluated at $s=i\tau$: its half-plane of convergence $\re s>0$ is $\Pi_-$, and the boundary value $(\tau-i0)^{1/3}$ in \eqref{eq:rho} corresponds to $\re s\to0^+$. This is also the reason why no problem backwards in time is ever solved (which would be the left half-line problem, with two boundary conditions): $u$ vanishes identically for $t\le0$. For $\tau\in\Pi_-$ the equation becomes the ODE $\pd_x^3\F_tu=-i\tau\,\F_tu$, whose solutions are combinations of $e^{ikx}$ with $(ik)^3=-i\tau$, i.e.\ $k^3=\tau$, $k\in\{\rho,\om\rho,\bar\om\rho\}$ with $\rho=\tau^{1/3}$ the cube root holomorphic in $\Pi_-$, $\arg\rho\in(-\frac\pi3,0)$. Then
$$\arg\rho\in(-\tfrac\pi3,0),\qquad \arg(\om\rho)\in(\tfrac\pi3,\tfrac{2\pi}3),\qquad \arg(\bar\om\rho)\in(-\pi,-\tfrac{2\pi}3),$$
so that $\ima\rho<0$, $\ima(\bar\om\rho)<0$ and $\ima(\om\rho)>0$: for every $\tau\in\Pi_-$, $e^{i\rho x}$ and $e^{i\bar\om\rho x}$ grow exponentially as $x\to+\infty$ and $e^{i\om\rho x}$ decays. Since $\F_tu(\cdot,\tau)$ is bounded on $x>0$ and equals $\F_th(\tau)=\widehat h(\tau)$ at $x=0$, we conclude
\begin{equation}\label{eq:BSZ}
\F_tu(x,\tau)=\widehat h(\tau)\,e^{i\om\rho(\tau)x},\qquad x>0,\ \tau\in\Pi_- ,
\end{equation}
and, letting $\ima\tau\uparrow0$, the same formula on $\R_\tau$ with $\rho(\tau)=(\tau-i0)^{1/3}$ as in \eqref{eq:rho}. This is the Laplace-transform representation of \cite{BSZ1}; it is also what the unified-transform formula of \cite{Fokas1,Fokas2} reduces to for zero initial data, when the contour $\partial D^+$ is collapsed onto the real axis (compare \cite[(2.36)]{Fokas2} and \cite[Remark 2.1]{HY}).

 Conversely, given any $h\in C^\infty_{0,c}(\R^+)$, the function defined by the right-hand side of \eqref{eq:BSZ} is, for each fixed $x>0$, holomorphic and bounded in $\tau\in\Pi_-$ (by \eqref{eq:rho-props} and the computation above, $\ima(\om\rho)>0$ persists in $\Pi_-$); hence $u(x,t)=\frac1{2\pi}\int e^{it\tau}\widehat h(\tau)e^{i\om\rho x}d\tau$ vanishes for $t<0$ (Paley--Wiener in $t$; the argument is written out in the proof of Proposition \ref{prop:L}(b)). For real $\tau<0$ the branch gives $\rho=|\tau|^{1/3}e^{-i\pi/3}$, so that the labels of the three roots rotate --- the real root becomes $\bar\om\rho$ --- while the admissible root $\om\rho$ is the same analytic branch for all $\tau$.

Up to here nothing has been chosen: \eqref{eq:BSZ} is the causal solution of the linear IBVP on $x>0$, and it vanishes identically for $t\le0$. The estimates of Section \ref{sec:spaces}, however, live on $\R_x\times\R_t$, so we must extend \eqref{eq:BSZ} to $x<0$, and the whole point is the choice of the extension. Extending by zero produces a jump at $x=0$, which costs $s+3b<\frac12$ in $X^{s,b}$; the Duhamel operators of \cite{CK,Holmer} extend by left-travelling free waves emitted from $x=0$, which cost $b<\frac12$ (Remark \ref{rem:holmer}). We extend instead by a decaying profile.

Now we extend the solution for $x<0$. Fix integers $m\ge0$, $n\ge m+2$, and look for
\begin{equation}\label{eq:ext}
\F_tu(x,\tau)=\widehat h(\tau)\,g\big(\rho(\tau)x\big),\qquad g(y):=\mathbf 1_{y>0}\,e^{i\om y}+\mathbf 1_{y<0}\,P(y)\,e^{y},
\end{equation}
with $P$ a polynomial of degree $\le n-1$ (for complex $\rho$ the right-hand side is to be read as $\mathbf 1_{x>0}e^{i\om\rho x}+\mathbf 1_{x<0}P(\rho x)e^{\rho x}$). By \eqref{eq:rho-props} both pieces decay away from $x=0$, and both remain holomorphic and bounded in $\tau\in\Pi_-$, so causality survives. The freedom in $P$ is used to prescribe the regularity of the matching at $x=0$ and the vanishing of low spatial frequencies.

\begin{lemma}\label{lem:glue}
Let $m\ge0$, $n\ge m+2$, and $c=-i\om^{-m}(\om+i)^n$. There is a unique polynomial $P$ of degree $\le n-1$ such that $g$ in \eqref{eq:ext} is of class $C^{n-m-1}$ on $\R$ and has vanishing moments $\int_\R y^jg(y)\,dy=0$ for $0\le j\le m-1$. For this $P$,
\begin{equation}\label{eq:ghat}
\widehat g(\eta)=\int_\R e^{-iy\eta}g(y)\,dy=\frac{a(\eta)}{\eta-\om},\qquad a(\eta)=c\,\frac{\eta^m}{(\eta+i)^n},
\end{equation}
and $P$ is determined by $\dfrac{Q(\eta)}{(\eta+i)^n}:=\dfrac{c\,\eta^m+i(\eta+i)^n}{(\eta-\om)(\eta+i)^n}=\displaystyle\sum_{j=0}^{n-1}(-1)^jj!\,i^{j+1}\frac{p_j}{(\eta+i)^{j+1}}$, where $P(y)=\sum_{j=0}^{n-1}p_jy^j$.
\end{lemma}

\begin{proof}
Since $\ima\om>0$, $\int_0^\infty e^{-iy\eta}e^{i\om y}dy=\frac{-i}{\eta-\om}$. For $j\ge0$, substituting $y=-z$ and using $1-i\eta=-i(\eta+i)$,
$$\int_{-\infty}^0e^{-iy\eta}y^je^{y}dy=\int_0^\infty e^{-(1-i\eta)z}(-z)^jdz=\frac{(-1)^jj!}{(1-i\eta)^{j+1}}=\frac{(-1)^jj!\,i^{j+1}}{(\eta+i)^{j+1}} .$$
Hence, for any $P$ of degree $\le n-1$, $\widehat g(\eta)=\frac{-i}{\eta-\om}+\frac{Q(\eta)}{(\eta+i)^n}$ with $Q$ a polynomial of degree $\le n-1$, and conversely every such $Q$ arises from exactly one $P$ (expand $Q$ in powers of $\eta+i$). Thus
$$\widehat g(\eta)=\frac{N(\eta)}{(\eta-\om)(\eta+i)^n},\qquad N(\eta):=-i(\eta+i)^n+Q(\eta)(\eta-\om),\quad \deg N\le n .$$
Now $g$ is bounded with compactly supported singularities, so $g\in C^{n-m-1}$ with piecewise smooth pieces is equivalent to $\widehat g(\eta)=O(|\eta|^{-(n-m+1)})$ at infinity, i.e.\ to $\deg N\le m$; and $\int y^jg=0$ for $j<m$ is equivalent to $\widehat g^{(j)}(0)=0$ for $j<m$, i.e.\ to $\eta^m\mid N$. Both together mean $N(\eta)=c'\eta^m$ for a constant $c'$, and evaluating at $\eta=\om$ (where $N(\om)=-i(\om+i)^n$) gives $c'=c$. Conversely, given $c$, $Q(\eta)=\big(c\eta^m+i(\eta+i)^n\big)/(\eta-\om)$ is a polynomial of degree $n-1$ precisely because the numerator vanishes at $\om$; this determines $Q$, hence $P$, uniquely.
\end{proof}

 Let $P$ be as in Lemma \ref{lem:glue} and $\F_tu$ as in \eqref{eq:ext}. For $\tau>0$, $\rho>0$ and the change of variables $y=\rho x$ gives $\int e^{-ix\xi}g(\rho x)dx=\rho^{-1}\widehat g(\xi/\rho)$. For $\tau<0$ the same computation as in the proof of Lemma \ref{lem:glue}, now with $\int_0^\infty e^{-ix\xi}e^{i\om\rho x}dx=\frac{-i}{\rho(\zeta-\om)}$ (convergent since $\ima\om\rho>0$) and $\int_{-\infty}^0e^{-ix\xi}(\rho x)^je^{\rho x}dx=\rho^{-1}(-1)^jj!\,i^{j+1}(\zeta+i)^{-j-1}$ (convergent since $\re\rho>0$), $\zeta=\xi/\rho$, gives the same algebraic identity. In both cases
\begin{equation}\label{eq:symbol-derived}
\int_\R e^{-ix\xi}\,\F_tu(x,\tau)\,dx=\widehat h(\tau)\,\frac{1}{\rho(\tau)}\,\frac{a(\zeta)}{\zeta-\om},\qquad \zeta=\frac{\xi}{\rho(\tau)} ,
\end{equation}
which is the symbol $\widehat h(\tau)K(\xi,\tau)$ of Definition \ref{def:L} below. The two parameters have a transparent meaning: $n-m-1$ is the number of derivatives of the matching at $x=0$ (it governs the decay of $K$ as $|\xi|\to\infty$ and hence the admissible range of $s$ from above), and $m$ is the number of vanishing moments of the layer (it governs the behaviour of $K$ as $\xi\to0$ and hence the admissible range of $s$ from below; for $s<-\frac12$ one needs $m\ge1$, see Case A in the proof of Proposition \ref{prop:L}(e)). In Holmer's family both roles are played by the single parameter $\la$ in $(\xi-i0)^{-\la}$.

\subsection{Definition}
Let $\om:=e^{2\pi i/3}$, so that $\bar\om=e^{4\pi i/3}$ and $1-\zeta^3=-(\zeta-1)(\zeta-\om)(\zeta-\bar\om)$. Fix integers
\begin{equation}\label{eq:mn}
m\ge0,\qquad n\ge m+4,
\end{equation}
and define the rational function
\begin{equation}\label{eq:a}
a(\zeta):=c\,\frac{\zeta^m}{(\zeta+i)^n},\qquad c:=-i\,\om^{-m}(\om+i)^n,\qquad\text{so that}\quad a(\om)=-i .
\end{equation}
With $\rho(\tau)=(\tau-i0)^{1/3}$ as in \eqref{eq:rho}, set, for $\tau\in\R\setminus\{0\}$ and $\xi\in\R$,
\begin{equation}\label{eq:K}
\zeta=\zeta(\xi,\tau):=\frac{\xi}{\rho(\tau)},\qquad K(\xi,\tau):=\frac{1}{\rho(\tau)}\,\frac{a(\zeta)}{\zeta-\om}.
\end{equation}

\begin{definition}\label{def:L}
For $h\in C^\infty_{0,c}(\R^+)$ define (by \eqref{eq:symbol-derived}, this is the function \eqref{eq:ext} of Section \ref{sec:derivation})
\begin{equation}\label{eq:Ldef}
\Lc h(x,t):=\F^{-1}\big[\widehat h(\tau)K(\xi,\tau)\big](x,t)=\frac{1}{(2\pi)^2}\iint_{\R^2}e^{i(x\xi+t\tau)}\,\widehat h(\tau)K(\xi,\tau)\,d\xi\,d\tau .
\end{equation}
\end{definition}

The following elementary facts are used throughout. Write $\rho(\tau)=|\tau|^{1/3}e^{i\theta_\tau}$ with $\theta_\tau=0$ for $\tau>0$ and $\theta_\tau=-\pi/3$ for $\tau<0$.

\begin{lemma}\label{lem:algebra}
Let $\tau\in\R\setminus\{0\}$.
\begin{enumerate}
\item[(i)] $\tau-\xi^3=\rho(\tau)^3(1-\zeta^3)=-\rho(\tau)^3(\zeta-1)(\zeta-\om)(\zeta-\bar\om)$ for all $\xi\in\Com$.
\item[(ii)] As $\xi$ runs over $\R$, $\zeta$ runs over the line $\Gamma_\tau:=e^{-i\theta_\tau}\R$, i.e.\ $\Gamma_\tau=\R$ if $\tau>0$ and $\Gamma_\tau=e^{i\pi/3}\R$ if $\tau<0$, and $d\xi=\rho(\tau)\,d\zeta$. As $\xi$ runs over the open upper half-plane, $\zeta$ runs over the open half-plane $H_\tau$ to the left of $\Gamma_\tau$ (oriented by increasing $\xi$), namely $H_\tau=\{\arg\zeta\in(-\theta_\tau,\pi-\theta_\tau)\}$.
\item[(iii)] $\dist(\om,\Gamma_\tau)=\frac{\sqrt3}2$ and $\om\in H_\tau$; $\dist(-i,\overline{H_\tau})=\dist(-i,\Gamma_\tau)\ge\frac12$, so $-i\notin\overline{H_\tau}$.
\item[(iv)] The real root of $\tau-\xi^3$ is $\xi=\tau^{1/3}$ (real cube root); it corresponds to $\zeta=1$ if $\tau>0$ and to $\zeta=\bar\om$ if $\tau<0$. In particular $K(\cdot,\tau)$ has no singularity on $\R_\xi$, and
\begin{equation}\label{eq:Kbound}
|K(\xi,\tau)|\le C_{m,n}\,|\tau|^{-1/3}\,|\zeta|^m\lan\zeta\ran^{-n}\qquad(\xi\in\R),
\end{equation}

\end{enumerate}
\end{lemma}
\begin{proof}
(i) follows from $\rho^3=\tau$. (ii): $\zeta=\xi/\rho$ is a rotation by $-\theta_\tau$ followed by a dilation; $H_\tau$ is the image of $\{\ima\xi>0\}$. (iii): $\arg\om=\frac{2\pi}3\in(0,\pi)\cap(\frac\pi3,\frac{4\pi}3)$, and the distance from $\om$ to the line through $0$ with direction $e^{-i\theta_\tau}$ is $|\sin(\frac{2\pi}3+\theta_\tau)|=\frac{\sqrt3}2$ in both cases; the distance from $-i=e^{-i\pi/2}$ to the same line is $|\sin(-\frac\pi2+\theta_\tau)|\in\{1,\frac12\}$, and $\arg(-i)=\frac{3\pi}2\notin[-\theta_\tau,\pi-\theta_\tau]$. (iv): for $\tau<0$, $\tau^{1/3}=-|\tau|^{1/3}$ and $\zeta=-e^{i\pi/3}=e^{i4\pi/3}=\bar\om$. For $\xi\in\R$ we have $\zeta\in\Gamma_\tau$, so $|\zeta-\om|\ge\frac{\sqrt3}2$ by (iii). Moreover $|\zeta+i|\ge\frac12$ on $\overline{H_\tau}$, whence $|\zeta+i|\ge\max(\frac12,|\zeta|-1)\ge\frac14\lan\zeta\ran$ and
\begin{equation}\label{eq:abound}
|a(\zeta)|\le C|\zeta|^m\lan\zeta\ran^{-n}\qquad\text{on }\overline{H_\tau};
\end{equation}
\eqref{eq:Kbound} follows.
\end{proof}

\begin{remark}
The role of the two parameters is the following. The factor $(\zeta+i)^{-n}$ makes $\Lc h$ smooth in $x$ across $x=0$ (it controls the decay of $K$ as $|\xi|\to\infty$); the factor $\zeta^m$ makes the symbol vanish at $\xi=0$ and controls the low spatial frequencies, which for $s<-\frac12$ would otherwise dominate the $X^{s,\frac12}$ norm (see Case A in the proof of Proposition \ref{prop:L}(e)); it plays the role of $(\xi-i0)^{-\la}$ in Holmer's family $\Lc^\la_\pm$. Contrary to those, $K$ has no singularity on the real characteristic curve.
\end{remark}

\subsection{Properties}

\begin{proposition}\label{prop:L}
Let $m,n$ satisfy \eqref{eq:mn} and $h\in C^\infty_{0,c}(\R^+)$. Then:
\begin{enumerate}
\item[(a)] (\emph{Regularity and time traces}) $\Lc h\in C(\R^2)$, and for every $\alpha\in\R$,
$$\sup_{x\in\R}\|\Lc h(x,\cdot)\|_{H^\alpha(\R_t)}\le C\|h\|_{H^\alpha(\R)},$$
with $x\mapsto\Lc h(x,\cdot)$ continuous into $H^\alpha(\R_t)$. Moreover $\Lc h(0,t)=h(t)$ for all $t\in\R$.
\item[(b)] (\emph{Causality}) $\Lc h(x,t)=0$ for $t\le0$; in particular $\Lc h(x,0)=0$.
\item[(c)] (\emph{equation on $x>0$}) $F:=(\pd_t+\pd_x^3)\Lc h$ is a continuous function on $\R^2$ with $F(x,t)=0$ for $x>0$. Hence $\Lc h$ solves $\pd_tu+\pd_x^3u=0$ in $(0,+\infty)\times\R$.
\item[(d)] (\emph{Explicit form on $x>0$}) For $x>0$,
$$\Lc h(x,t)=\frac1{2\pi}\int_\R e^{it\tau}\,\widehat h(\tau)\,e^{ix\om\rho(\tau)}\,d\tau,\qquad |e^{ix\om\rho(\tau)}|=e^{-\frac{\sqrt3}2|\tau|^{1/3}x}.$$
\item[(e)] (\emph{Bourgain--Besov norm}) Let $s\ge-1$ and assume in addition
\begin{equation}\label{eq:mn2}
m>-s-\tfrac12,\qquad n>m+s+2,\qquad \supp h\subset[0,2].
\end{equation}
Then
$$\|\Lc h\|_{X^{s,\frac12,1}}\le C\|h\|_{H^{\frac{s+1}3}(\R)},\qquad\text{hence}\qquad \|\Lc h\|_{C(\R_t;H^s(\R_x))}\le C\|h\|_{H^{\frac{s+1}3}(\R)} .$$
\end{enumerate}
For $s=-\frac34$ one may take $m=1$, $n=5$, and then $\|\Lc h\|_X\le2\|\Lc h\|_{X^{-\frac34,\frac12,1}}\le C\|h\|_{H^{1/12}}$.
\end{proposition}

\begin{proof}
Since $\widehat h\in\mathcal S(\R)$ and, by \eqref{eq:Kbound} and the change of variables $\xi=\rho\zeta$,
\begin{equation}\label{eq:L1}
\int_\R|K(\xi,\tau)|\,d\xi\le C|\tau|^{-1/3}\int_{\Gamma_\tau}|\zeta|^m\lan\zeta\ran^{-n}|\rho|\,|d\zeta|=C_1<\infty
\end{equation}
uniformly in $\tau\ne0$ (because $n\ge m+2$), the integrand in \eqref{eq:Ldef} is in $L^1(\R^2)$ and $\Lc h$ is a bounded continuous function.

\emph{(a)} Define
\begin{equation}\label{eq:kappa}
\kappa(x,\tau):=\frac1{2\pi}\int_\R e^{ix\xi}K(\xi,\tau)\,d\xi=\frac1{2\pi}\int_{\Gamma_\tau}e^{ix\rho(\tau)\zeta}\,\frac{a(\zeta)}{\zeta-\om}\,d\zeta ,
\end{equation}
where we used $d\xi=\rho\,d\zeta$ and $\rho\cdot\rho^{-1}=1$. On $\Gamma_\tau$, $\rho\zeta=\xi\in\R$, so $|e^{ix\rho\zeta}|=1$ and by Lemma \ref{lem:algebra}(iii) $|\kappa(x,\tau)|\le C_0$ for all $(x,\tau)$. By Fubini, $$\Lc h(x,t)=\frac1{2\pi}\int e^{it\tau}\widehat h(\tau)\kappa(x,\tau)d\tau,$$ then  $\mathcal F_t\Lc h(x,\cdot)=\widehat h(\tau)\kappa(x,\tau)$, whence $\|\Lc h(x,\cdot)\|_{H^\alpha}\le C_0\|h\|_{H^\alpha}$. Continuity in $x$: $\kappa(x,\tau)\to\kappa(x_0,\tau)$ for each $\tau$ by dominated convergence in \eqref{eq:kappa}, and then $\int\lan\tau\ran^{2\alpha}|\widehat h|^2|\kappa(x,\tau)-\kappa(x_0,\tau)|^2d\tau\to0$ by dominated convergence again.

For the trace, we compute $\kappa(0,\tau)$ by residues. The function $\zeta\mapsto a(\zeta)/(\zeta-\om)$ is meromorphic with a simple pole at $\om$ and a pole at $-i$, and is $O(|\zeta|^{m-n-1})=O(|\zeta|^{-2})$ at infinity. By Lemma \ref{lem:algebra}(iii), $\om\in H_\tau$ and $-i\notin\overline{H_\tau}$. Closing $\Gamma_\tau$ by a large arc in $H_\tau$ (which lies to the left of the oriented line, so the contour is positively oriented) we get
$$\kappa(0,\tau)=\frac1{2\pi}\,2\pi i\,\Res_{\zeta=\om}\frac{a(\zeta)}{\zeta-\om}=i\,a(\om)=1 .$$
Hence $\Lc h(0,t)=\frac1{2\pi}\int e^{it\tau}\widehat h(\tau)d\tau=h(t)$.

\emph{(b)} Fix $\xi\in\R$. For $z$ in the open lower half-plane $\Pi_-=\{\ima z<0\}$ put $\rho(z)=z^{1/3}$, $\zeta=\xi/\rho(z)$ and $K(\xi,z)=\rho(z)^{-1}a(\zeta)/(\zeta-\om)$. Since $\arg\rho(z)\in(-\frac\pi3,0)$, we have $\arg\zeta\in(0,\frac\pi3)$ if $\xi>0$ and $\arg\zeta\in(\pi,\frac{4\pi}3)$ if $\xi<0$; in both cases $\zeta\ne\om,-i$, so $z\mapsto K(\xi,z)$ is holomorphic in $\Pi_-$, continuous on $\overline{\Pi_-}\setminus\{0\}$, and $|K(\xi,z)|\le C|z|^{-1/3}$ (because $|\zeta|^m\lan\zeta\ran^{-n}\le1$ and the distances of $\om$, $-i$ to the sectors above are positive). Since $\supp h\subset[0,\infty)$, $\widehat h(z)=\int_0^\infty e^{-itz}h(t)dt$ is holomorphic in $\Pi_-$, continuous on $\overline{\Pi_-}$, and integrating by parts $N$ times, $|\widehat h(z)|\le C_N(1+|z|)^{-N}$ on $\overline{\Pi_-}$. For $t\le0$ and $z\in\overline{\Pi_-}$, $|e^{itz}|=e^{-t\ima z}\le1$. Apply Cauchy's theorem to $z\mapsto e^{itz}\widehat h(z)K(\xi,z)$ on the region $\{\ve<|z|<R,\ \ima z<0\}$: the integral over the large semicircle is $O(R\cdot R^{-N}R^{-1/3})\to0$, that over the small semicircle is $O(\ve\cdot\ve^{-1/3})\to0$, and we conclude $\int_\R e^{it\tau}\widehat h(\tau)K(\xi,\tau)d\tau=0$ for every $\xi$ and every $t\le0$. Integrating in $\xi$ (Fubini) gives $\Lc h(x,t)=0$ for $t\le0$.

\emph{(c)} Since $(\pd_t+\pd_x^3)^\wedge=i(\tau-\xi^3)$, Lemma \ref{lem:algebra}(i) and $\rho^3/\rho=\rho^2=(\tau-i0)^{2/3}$ give
\begin{equation}\label{eq:Fhat}
\begin{split}
&\widehat F(\xi,\tau)=i(\tau-\xi^3)\widehat h(\tau)K(\xi,\tau)=-i\,\widehat h(\tau)\,(\tau-i0)^{2/3}\,b(\zeta),\\& b(\zeta):=a(\zeta)(\zeta-1)(\zeta-\bar\om)=c\,\frac{\zeta^m(\zeta-1)(\zeta-\bar\om)}{(\zeta+i)^n}.
\end{split}
\end{equation}
By Lemma \ref{lem:algebra}(iii), $|\zeta+i|\ge\frac14\lan\zeta\ran$ on $\overline{H_\tau}$, hence $|b(\zeta)|\le C\lan\zeta\ran^{m+2-n}\le C\lan\zeta\ran^{-2}$ on $\overline{H_\tau}$. Consequently $\widehat F(\cdot,\tau)\in L^1(\R)$ with $$\int|\widehat F(\xi,\tau)|d\xi\le C|\widehat h(\tau)||\tau|^{2/3}|\rho(\tau)|=C|\widehat h(\tau)||\tau|,$$ so $\widehat F\in L^1(\R^2)$ and $F=\F^{-1}\widehat F$ is a bounded continuous function with $(\pd_t+\pd_x^3)\Lc h=F$ in $\mathcal S'(\R^2)$. Fix $\tau\ne0$ and $x>0$. The function $\xi\mapsto\widehat F(\xi,\tau)$ extends holomorphically to $\{\ima\xi>0\}$ (as $\zeta=\xi/\rho$ then lies in $H_\tau$, where $b$ is holomorphic), with $|\widehat F(\xi+i\eta,\tau)|\le C|\widehat h(\tau)||\tau|^{2/3}\lan(\xi+i\eta)/\rho\ran^{-2}\le C|\widehat h(\tau)||\tau|^{2/3}\lan\xi/\rho\ran^{-2}$ uniformly in $\eta\ge0$. By Cauchy's theorem applied to the rectangle with vertices $\pm R$, $\pm R+i\eta$, letting $R\to\infty$ (the vertical sides are $O(R^{-2})$),
\begin{equation*}\begin{split}
\int_\R e^{ix\xi}\widehat F(\xi,\tau)d\xi&=\int_\R e^{ix(\xi+i\eta)}\widehat F(\xi+i\eta,\tau)d\xi,\\
\Big|\int_\R e^{ix(\xi+i\eta)}\widehat F(\xi+i\eta,\tau)d\xi\Big|&\le e^{-x\eta}\,C|\widehat h(\tau)||\tau| .
\end{split}\end{equation*}
Letting $\eta\to+\infty$ we get $\int_\R e^{ix\xi}\widehat F(\xi,\tau)d\xi=0$ for $x>0$, $\tau\ne0$, and therefore $F(x,t)=\frac1{(2\pi)^2}\int e^{it\tau}\big(\int e^{ix\xi}\widehat F(\xi,\tau)d\xi\big)d\tau=0$ for $x>0$.

\emph{(d)} For $x>0$ and $\zeta\in\overline{H_\tau}$, $\rho\zeta$ lies in the closed upper half-plane, so $|e^{ix\rho\zeta}|\le1$; as $a(\zeta)/(\zeta-\om)=O(|\zeta|^{-2})$, we may close $\Gamma_\tau$ in \eqref{eq:kappa} by a large arc in $H_\tau$ and obtain $\kappa(x,\tau)=i\,a(\om)e^{ix\rho\om}=e^{ix\om\rho(\tau)}$. Finally $\ima(\om\rho(\tau))=|\tau|^{1/3}\sin(\frac{2\pi}3+\theta_\tau)=\frac{\sqrt3}2|\tau|^{1/3}$ in both cases $\theta_\tau\in\{0,-\frac\pi3\}$.

\emph{(e)} Write $W:=\widehat h\,K=\widehat{\Lc h}$ and, for $l\ge1$, $S_l:=\{\tau:2^l\le|\tau|<2^{l+1}\}$, $S_0:=\{|\tau|<2\}$, $\widehat h_l:=\mathbf 1_{S_l}\widehat h$. Set $\beta:=\frac{s+1}3$ and $c_l:=2^{\beta l}\|\widehat h_l\|_{L^2}$, so that $\sum_{l\ge0}c_l^2\le C\|h\|^2_{H^\beta}$. For $j,k\ge0$ and $\tau\in\R$ let $B_k(\tau):=\{\xi:(\tau,\xi)\in B_k\}$ and
$$M_{jk}(\tau):=\|\lan\xi\ran^s\lan\sig\ran^{\frac12}K(\cdot,\tau)\|_{L^2_\xi(A_j\cap B_k(\tau))}.$$
Since the sets $A_j\cap B_k\cap(\R\times S_l)$, $l\ge0$, are disjoint and cover $A_j\cap B_k$, the triangle inequality in $L^2$ gives $\|F\|_{L^2(A_j\cap B_k)}\le\sum_{l\ge0}\|F\|_{L^2(A_j\cap B_k\cap(\R\times S_l))}$ for every $F$; and since $|W|=|\widehat h_l||K|$ on $\R\times S_l$, $$\|\lan\xi\ran^s\lan\sig\ran^{1/2}W\|_{L^2(A_j\cap B_k\cap(\R\times S_l))}\le\|\widehat h_l\|_{L^2}\sup_{\tau\in S_l}M_{jk}(\tau).$$ Summing over $k$, separating $l=0$ from $l\ge1$ and using $(a+b)^2\le2a^2+2b^2$, we have
\begin{equation}\label{eq:assemble}
\|\Lc h\|_{X^{s,\frac12,1}}^2\le 2\sum_{j\ge0}\Big(\sum_{l\ge1}\|\widehat h_l\|_{L^2}\,\Theta_{jl}\Big)^2+2\sum_{j\ge2}\|\widehat h_0\|_{L^2}^2\,\Theta_{j0}^2+2R_0,\qquad \Theta_{jl}:=\sum_{k\ge0}\sup_{\tau\in S_l}M_{jk}(\tau),
\end{equation}
where $R_0:=\sum_{j\in\{0,1\}}\big(\sum_k\|\lan\xi\ran^s\lan\sig\ran^{\frac12}W\|_{L^2(A_j\cap B_k\cap(\R\times S_0))}\big)^2$ collects the two blocks $j\in\{0,1\}$, $l=0$, for which $\sup_{\tau\in S_0}M_{jk}(\tau)=+\infty$ (because $|K(\cdot,\tau)|\sim|\tau|^{-1/3}$ on $|\xi|\lesssim|\tau|^{1/3}$) and which are therefore estimated directly below. We bound $\Theta_{jl}$ using \eqref{eq:Kbound} and $|\zeta|=|\xi||\tau|^{-1/3}$. Recall $2^{j-1}\le|\xi|<2^{j+1}$ on $A_j$, $j\ge1$, and $|\xi|<\sqrt3$ on $A_0$. Let $\delta_1:=\frac{2s+2m+1}6>0$ and $\delta_2:=\frac{n-m-s-2}3>0$ by \eqref{eq:mn2}.

\emph{Case A: $l\ge3j+4$.} Then $|\xi|^3<2^{3j+3}\le\frac12|\tau|$ on $A_j$, so $\frac12|\tau|\le|\sig|\le\frac32|\tau|$ and $M_{jk}(\tau)=0$ unless $|k-l|\le2$. Also $|\zeta|\le 2^{j+1}2^{-l/3}\le1$, so $|K|\le C|\tau|^{-1/3}(2^{j+1}2^{-l/3})^m$ and, using $|A_j|\le2^{j+2}$,
$$M_{jk}(\tau)\le C\,2^{sj}\,2^{l/2}\,2^{-l/3}\,2^{(j-l/3)m}\,2^{j/2}=C\,2^{(s+m+\frac12)j}\,2^{(\frac16-\frac m3)l}=C\,2^{\beta l}\,2^{-\delta_1(l-3j)} ,$$
where the last identity is checked by comparing exponents: $$(s+m+\frac12)j+(\frac16-\frac m3)l-\beta l=-\frac{2s+2m+1}{6}(l-3j).$$ Hence $\Theta_{jl}\le C2^{\beta l}2^{-\delta_1(l-3j)}$.

\emph{Case B: $1\le l\le3j-5$.} Then $|\tau|<2^{l+1}\le2^{3j-4}\le\frac12|\xi|^3$ on $A_j$, so $\frac12|\xi|^3\le|\sig|\le\frac32|\xi|^3$ and $M_{jk}(\tau)=0$ unless $|k-3j|\le4$. Also $|\zeta|\ge2^{j-1}2^{-(l+1)/3}\ge1$, so $|K|\le C|\tau|^{-1/3}|\zeta|^{m-n}\le C\,2^{(m-n)j}\,2^{(n-m-1)l/3}$, and, with $\lan\sig\ran^{1/2}\le C2^{3j/2}$,
\begin{equation*}\begin{split}M_{jk}(\tau)&\le C\,2^{sj}\,2^{3j/2}\,2^{(m-n)j}\,2^{(n-m-1)l/3}\,2^{j/2}\\&=C\,2^{(s+2+m-n)j}\,2^{(n-m-1)l/3}=C\,2^{\beta l}\,2^{-\delta_2(3j-l)},\end{split}\end{equation*}
since $(s+2+m-n)j+\frac{n-m-1}3l-\beta l=-\frac{n-m-s-2}3(3j-l)$. Hence $\Theta_{jl}\le C2^{\beta l}2^{-\delta_2(3j-l)}$.

\emph{Case C: $3j-5<l<3j+4$, $l\ge1$} (at most $8$ values of $l$ for each $j$, and $j\ge1$ here unless $l\le3$). On $A_j\times S_l$ we have $\lan\sig\ran\le1+|\tau|+|\xi|^3<1+2^{l+1}+2^{3j+3}<2^{l+8}$ (as $3j+3\le l+7$), so $M_{jk}=0$ for $k>l+7$; $|K|\le C|\tau|^{-1/3}\le C2^{-l/3}$ by \eqref{eq:Kbound} and $m\le n$; and for $j\ge1$, since $\xi\mapsto\tau-\xi^3$ is monotone on $A_j\cap\{\pm\xi>0\}$ with $|\pd_\xi(\tau-\xi^3)|\ge3\cdot2^{2j-2}$, the set $A_j\cap B_k(\tau)$ (on which $|\sig|$ lies in an interval of length $2^{k+2}$) has measure $\le C2^{k-2j}$. Therefore
\begin{equation*}\begin{split}M_{jk}(\tau)&\le C\,2^{sj}\,2^{k/2}\,2^{-l/3}\,(2^{k-2j})^{1/2}=C\,2^{(s-1)j}\,2^{k}\,2^{-l/3},\\
\Theta_{jl}&\le C\,2^{(s-1)j}\,2^{-l/3}\sum_{k\le l+7}2^k\le C\,2^{(s-1)j}\,2^{2l/3}\le C\,2^{\beta l},\end{split}\end{equation*}
using $2^j\sim2^{l/3}$. For $j=0$ and $l\le3$ the same bound holds trivially since $|K|\le C|\tau|^{-1/3}\le C$ on $S_l$, $l\ge1$, and $k\le 10$.

\emph{Low temporal frequencies: $l=0$.} For $j\ge2$ and $|\tau|<2$ we are in the situation of Case B ($|\tau|<2\le\frac12|\xi|^3$), and $|K|\le C|\tau|^{-1/3}(|\xi||\tau|^{-1/3})^{m-n}=C|\xi|^{m-n}|\tau|^{(n-m-1)/3}\le C|\xi|^{m-n}$ because $n-m-1\ge0$; thus $\Theta_{j0}\le C2^{(s+2+m-n)j}$, which is square-summable in $j$ by \eqref{eq:mn2}. For $j\in\{0,1\}$ and $|\tau|<2$ we have $|\xi|<\sqrt{15}$, hence $\lan\sig\ran<1+2+15^{3/2}<64$ and only $k\le5$ occur; thus the $\ell^1_k$ sum in $R_0$ has at most six terms and $\sum_{k\le5}\|F\|_{L^2(A_j\cap B_k\cap(\R\times S_0))}\le6\|F\|_{L^2(A_j\times S_0)}$, so that $R_0\le C\int_{S_0}|\widehat h(\tau)|^2\int_{|\xi|<\sqrt{15}}|K(\xi,\tau)|^2d\xi\,d\tau$. By \eqref{eq:Kbound} and the change of variables $\xi=\rho\zeta$,
$$\int_{\R}|K(\xi,\tau)|^2d\xi\le C|\tau|^{-2/3}\int_{\Gamma_\tau}|\zeta|^{2m}\lan\zeta\ran^{-2n}|\rho|\,|d\zeta|=C|\tau|^{-1/3}.$$
Hence $$R_0\leq C\int_{|\tau|<2}|\widehat h(\tau)|^2|\tau|^{-1/3}d\tau\le C\|\widehat h\|^2_{L^\infty}\le C\|h\|^2_{L^1}\le C\|h\|_{L^2}^2,$$ where we used $\supp h\subset[0,2]$; and $\sum_{j\ge2}\|\widehat h_0\|^2_{L^2}\Theta_{j0}^2\le C\sum_{j\ge2}2^{2(s+2+m-n)j}\|h\|^2_{L^2}\le C\|h\|_{L^2}^2$.

\emph{Summation.} Inserting the bounds of Cases A--C into \eqref{eq:assemble} for $l\ge1$, we obtain
$$\sum_{j\ge0}\Big(\sum_{l\ge1}\|\widehat h_l\|_{L^2}\Theta_{jl}\Big)^2\le C\sum_{j\ge0}\Big(\sum_{l\ge1}\theta_{jl}\,c_l\Big)^2,\qquad
\theta_{jl}:=\begin{cases}2^{-\delta_1(l-3j)},& l\ge3j+4,\\ 2^{-\delta_2(3j-l)},& l\le3j-5,\\ 1,&\text{otherwise.}\end{cases}$$
Since $\sup_l\sum_j\theta_{jl}+\sup_j\sum_l\theta_{jl}\le C(\delta_1,\delta_2)$, then $$\sum_j(\sum_l\theta_{jl}c_l)^2\le C\sum_lc_l^2\le C\|h\|^2_{H^\beta}.$$ Together with the $l=0$ contribution this proves $\|\Lc h\|_{X^{s,\frac12,1}}\le C\|h\|_{H^\beta}$ (note $\|h\|_{L^2}\le\|h\|_{H^\beta}$ since $\beta\ge0$). The $C_tH^s$ bound follows from Lemma \ref{lem:embedding}. For $s=-\frac34$, \eqref{eq:mn} and \eqref{eq:mn2} read $m>\frac14$, $n>m+\frac54$, $n\ge m+4$, satisfied by $m=1$, $n=5$.
\end{proof}

\begin{remark}\label{rem:support}
The support assumption on $h$ in (e) is not a technicality: $\Lc$ is not bounded from $H^\beta(\R)$ into $X^{s,\frac12}$. Indeed $\int_\R|K(\xi,\tau)|^2d\xi=C|\tau|^{-1/3}$, so for $\widehat h=\ve^{-1/2}\mathbf 1_{|\tau|<\ve}$ one has $\|h\|_{H^\beta}\sim1$ while $\|\Lc h\|_{X^{s,\frac12}}\gtrsim\ve^{-1/6}$. The mechanism is physical: a boundary datum of temporal frequency $\tau$ produces a layer of width $|\tau|^{-1/3}$, which for $|\tau|\to0$ is not square integrable in $x$ uniformly. Since the fixed point only requires boundary values on a bounded time interval, restricting to $\supp h\subset[0,2]$ costs nothing (Section \ref{sec:proof}).
\end{remark}

\begin{corollary}\label{cor:Lext}
Let $s\in[-1,\frac12)$, $\beta=\frac{s+1}3\in[0,\frac12)$, and $m,n$ as in \eqref{eq:mn}, \eqref{eq:mn2}. The operator $\Lc$ extends uniquely to a bounded operator from $\{h\in H^\beta(\R):\supp h\subset[0,2]\}$ into $X^{s,\frac12,1}\cap C(\R_t;H^s(\R_x))\cap C(\R_x;H^\beta(\R_t))$, and for every such $h$: $\Lc h(\cdot,0)=0$ in $H^s(\R)$; $\Lc h(0,\cdot)=h$ in $H^\beta(\R)$; and $(\pd_t+\pd_x^3)\Lc h=0$ in the sense of distributions on $(0,+\infty)\times\R$.
\end{corollary}
\begin{proof}
For $\beta<\frac12$ the space $C^\infty_{c}((0,2))$ is dense in $\{h\in H^\beta:\supp h\subset[0,2]\}$ (\cite[Section 4]{Holmer}, Lemma \ref{lem:JK}). The bounds of Proposition \ref{prop:L}(a),(e) give the extension; (b), (a) and (c) pass to the limit in $C_tH^s$, in $C_xH^\beta_t$ and in $\mathcal D'((0,\infty)\times\R)$ respectively.
\end{proof}

\begin{remark}[Comparison with the operators of \cite{CK,Holmer}, \cite{BSZ2} and \cite{ET,HY,CT,GK}]\label{rem:holmer}
All boundary forcing operators in the literature agree with $\Lc$ on $x>0$, by uniqueness for the linear problem (compare \cite[Lemma 8.1]{Holmer}): there they reduce to the Laplace-transform formula of Proposition \ref{prop:L}(d). They differ only in the continuation to $x<0$, and this is what decides the admissible $b$.

\emph{Point sources.} Holmer's $\Lc^\la_+$ \cite[Section 3]{Holmer} (and the operator of \cite{CK}) is a Duhamel integral of a source $x_-^{\la-1}/\Gamma(\la)$ at the boundary --- the classical boundary potential of Cattabriga \cite{Cat}, used by Faminskii \cite{Fa2} in Bourgain-type spaces to obtain global well-posedness for $s\ge0$. Its space-time Fourier transform is, up to a free solution, $(\xi-i0)^{-\la}(\tau-i0)^{\frac\la3+\frac23}\widehat h(\tau)(\tau-\xi^3)^{-1}$ (proof of Lemma 5.8(d) there): a simple pole at the real root $\xi=\tau^{1/3}$ with nonvanishing residue, so that $\int\lan\xi\ran^{2s}\lan\sig\ran^{2b}|\cdot|^2d\xi$ converges only for $b<\frac12$ and the $\ell^1_k$ sum in \eqref{eq:Xsb1} diverges at $b=\frac12$. In physical space, $\Lc^\la_+h$ contains left-travelling free waves emitted from $x=0$. In \eqref{eq:K} the real root is not a pole and the trace is carried by the complex root $\om\rho(\tau)$.

\emph{Reflections.} The operators $BI^{m1},BI^{m2}$ of \cite{BSZ2} extend the solution on $x>0$ by even and odd reflections, combined through the cutoff $\Theta(\xi,\tau)=\chi(|\xi|-\delta|\tau|^{1/3})$ and a corrector $\omega(\tau)$ fixed by the implicit relation (2.22) there. Our computation in the proof of Proposition \ref{prop:L}(e) reproduces their results exactly:
\begin{enumerate}
\item[(i)] At fixed $\tau$, for $s>-\frac12$ and all $b$ for which the integral converges, $$\int\lan\xi\ran^{2s}\lan\sig\ran^{2b}|K|^2d\xi\sim|\tau|^{\frac{2s-1}3+2b},$$ i.e.\ $\|\Lc h\|_{X^{s,b}}\sim\|h\|_{H^{\frac{3b+s-1/2}3}}$ --- the regularity of Theorem 1.1 in \cite{BSZ2}, sharp in $b$, with $b=\frac12$ the value at which the datum has its natural regularity $\frac{s+1}3$. Convergence at $|\xi|\to\infty$ requires $n-m>s+2$, i.e.\ a $C^{n-m-1}$ junction at $x=0$ (Lemma \ref{lem:glue}); this is the role played in $BI^{m1}$ by the corrector $\omega(\tau)$, which removes a jump $O(\xi^{-1})$, and the restriction $b<\frac12-\frac s3$ there.
\item[(ii)] For $s<-\frac12$ and $m=0$ the region $|\xi|\le1$ dominates and yields only $\|h\|_{H^{\frac{3b-1}3}}$ --- the estimate of Appendix II of \cite{BSZ2} for the even extension $BI^e$, and the reason $BI^{m1}$ is introduced there. In $BI^{m1}$ the even part is switched off for $|\xi|\lesssim\delta|\tau|^{1/3}$ by $\Theta$; in $\Lc$ the same effect is produced by the factor $\zeta^m$, i.e.\ by the $m$ vanishing moments of the profile $g$ in \eqref{eq:ext}.
\end{enumerate}
Thus, for $-\frac34<s\le0$, Theorem 1.1 of \cite{BSZ2} and Proposition \ref{prop:L}(e) with $\ell^2_k$ in place of $\ell^1_k$ are the same statement. What \cite{BSZ2} does not address is the $\ell^1_k$ summability, which in the proof of Proposition \ref{prop:L}(e) comes from the pointwise bound \eqref{eq:Kbound} alone; the rational form of $K$ is what makes \eqref{eq:Kbound}, the trace identity and the support property elementary consequences of residue calculus.

\emph{Cutoffs.} A third family, going back to \cite{ET}, multiplies the decaying mode by a smooth cutoff $\rho(a_I\xi x)$ at the scale of the layer: Himonas--Yan \cite{HY} for KdV with Robin data, Compaan--Tzirakis \cite{CT} for the quartic gKdV, Gallego--Kwak \cite{GK} for higher-order KdV. The symbol is $|\tau|^{-1/3}\widehat\eta(\xi/(a_I|\tau|^{1/3}))\widehat h(\tau)$ with $\eta\in\mathcal S(\R)$ (times $\xi$ for Robin data): the structure of \eqref{eq:K} with a Schwartz profile in place of the rational one, and $|\widehat\eta(\xi/(a_I|\tau|^{1/3}))|\le C_N(|\tau|^{1/3}/(|\xi|+|\tau|^{1/3}))^N$ \cite[Lemma 2.2]{HY} is the counterpart of \eqref{eq:Kbound}. The resulting $X^{s,b}$ bounds \cite[Lemma 2.3]{HY}, \cite[Lemma 3.6]{CT}, \cite[Lemma 4.3]{GK} have the exponent of (i). The Besov norm is not considered in these works; the nonlinear arguments of \cite{HY,GK} stay at $b<\frac12$ with the correction $D_\alpha$ and the temporal space $Y^{s,-b}$ of \cite{Holmer}, and that of \cite{CT} places a linear part in $X^{s_0,b}$, $b>\frac12$, $s_0<s$, recovering the loss by nonlinear smoothing, which vanishes for KdV at $s=-\frac34$.

\end{remark}

\section{Proof of Theorem \ref{thm:main}}\label{sec:proof}

Throughout $s=-\frac34$, $\ga=\frac1{12}$, and $\Lc$ is the operator of Definition \ref{def:L} with $m=1$, $n=5$. Let $\psi\in C_c^\infty(\R)$, $\psi=1$ on $[-1,1]$, $\supp\psi\subset[-2,2]$, and let $\psi_2\in C^\infty_c(\R)$ with $\psi_2=1$ on $[0,1]$ and $\supp\psi_2\subset[-1,2]$.

Let $Z:=X_1\cap C([-1,1];H^{-3/4}(\R_x))\cap C(\R^+;H^{1/12}(\R_t))$ with norm $$\|u\|_Z:=\|u\|_{X_1}+\|u\|_{C([0,1];H^{-3/4}(\R^+))}+\|u\|_{C(\R^+;H^{1/12}(0,1))}.$$ For $u\in X_1$ choose an extension $\tilde u\in X$ with $\|\tilde u\|_X\le2\|u\|_{X_1}$; the restriction of $\Kc\pd_x(\tilde u^2)$ to $|t|\le1$ does not depend on the extension, and we simply write $\Kc\pd_x(u^2)$.

Given $u_0\in H^{-3/4}(\R^+)$ and $f\in H^{1/12}(\R^+)$, fix extensions $\tilde u_0\in H^{-3/4}(\R)$, $\tilde f\in H^{1/12}(\R)$ with norms at most twice those of $u_0$, $f$. For $u\in Z$ define
\begin{align}
h(u)&:=\chi_{(0,+\infty)}(t)\,\psi_2(t)\Big[\tilde f(t)-\psi(t)e^{-t\pd_x^3}\tilde u_0\big|_{x=0}+\psi(t)\Kc\pd_x(u^2)\big|_{x=0}\Big],\label{eq:h}\\
\Lambda u&:=\psi(t)e^{-t\pd_x^3}\tilde u_0-\psi(t)\Kc\pd_x(u^2)+\Lc h(u).\label{eq:Lambda}
\end{align}
The traces at $x=0$ in \eqref{eq:h} are those of Lemma \ref{lem:KPV} and Lemma \ref{lem:trace}, both continuous in $x$ with values in $H^{1/12}(\R_t)$. Note that $\supp h(u)\subset[0,2]$.

By Lemma \ref{lem:JK} ($\frac1{12}<\frac12$) and Lemma \ref{lem:cutoff}, multiplication by $\chi_{(0,\infty)}\psi_2$ is bounded on $H^{1/12}(\R)$. By Lemma \ref{lem:KPV}, $\|\psi e^{-t\pd_x^3}\tilde u_0|_{x=0}\|_{H^{1/12}}\le C\|\tilde u_0\|_{H^{-3/4}}$. By Lemma \ref{lem:trace} and Proposition \ref{prop:kishimoto}, $\|\psi\Kc\pd_x(\tilde u^2)|_{x=0}\|_{H^{1/12}}\le C\|\pd_x(\tilde u^2)\|_{\mathcal N}\le C\|\tilde u\|_X^2\le C\|u\|^2_{X_1}$. Thus we have that 
\begin{equation}\label{lem:hbound_1}
\|h(u)\|_{H^{1/12}(\R)}\le C\big(\|f\|_{H^{1/12}(\R^+)}+\|u_0\|_{H^{-3/4}(\R^+)}+\|u\|_Z^2\big).
\end{equation}

Similarly,  we obtain the difference estimate, writing $u^2-v^2=(u-v)(u+v)$,
\begin{equation}\label{lem:hbound_2}
\|h(u)-h(v)\|_{H^{1/12}(\R)}\le C\|u-v\|_Z\big(\|u\|_Z+\|v\|_Z\big).
\end{equation}

By Proposition \ref{prop:L}(e) and Corollary \ref{cor:Lext} (with $s=-\frac34$, $\supp h(u)\subset[0,2]$) and \eqref{lem:hbound_1}, \begin{equation}\label{esti_lh}
\begin{split}
\|\Lc h(u)\|_Z&\le\|\Lc h(u)\|_{X}+\|\Lc h(u)\|_{C_tH^{-3/4}}\le C\|h(u)\|_{H^{1/12}}\\&\le C\big(\|f\|_{H^{1/12}(\R^+)}+\|u_0\|_{H^{-3/4}(\R^+)}+\|u\|_Z^2\big).
\end{split}
\end{equation}

Thus, by using Lemma \ref{lem:kishimoto-linear}, Proposition \ref{prop:kishimoto} and \eqref{esti_lh} we get

\begin{equation}\label{cont_1}
\|\Lambda u\|_Z\le C\big(\|u_0\|_{H^{-3/4}(\R^+)}+\|f\|_{H^{1/12}(\R^+)}+\|u\|_Z^2\big)\end{equation}

Similarly, using the linearity of $\Lc$ and \eqref{lem:hbound_2}, \begin{equation}\label{cont2}\|\Lambda u-\Lambda v\|_Z\le C\|u-v\|_Z\big(\|u\|_Z+\|v\|_Z\big).\end{equation}

By using \eqref{cont_1} and \eqref{cont2}, if $$\|u_0\|_{H^{-3/4}(\R^+)}+\|f\|_{H^{1/12}(\R^+)}\le\delta$$ with $\delta$ small enough, $\Lambda$ maps the closed ball of radius $2C\delta$ in $Z$ into itself and is a contraction there. Let $u\in Z$ be its fixed point. The map $(u_0,f)\mapsto u$ is Lipschitz from $B_\delta:=\{\|u_0\|+\|f\|\le\delta\}$ into $Z$ by the standard argument, and analytic, since $\Lambda$ is a polynomial in $u$ with linear dependence on the data.

\medskip\noindent Finally, we claim that $u$ is a solution in the sense of Definition \ref{def:solution} with $T=1$.
(a) $u\in X_1$ by construction. (b) On $\R\times(-1,1)$, where $\psi=1$, $(\pd_t+\pd_x^3)\big(\psi e^{-t\pd_x^3}\tilde u_0\big)=0$ and $(\pd_t+\pd_x^3)\big(\psi\Kc\pd_x(\tilde u^2)\big)=\pd_x(\tilde u^2)=\pd_x(u^2)$ in $\mathcal D'(\R\times(-1,1))$ (the identity $(\pd_t+\pd_x^3)\Kc G=G$ holds in $\mathcal S'$ for Schwartz $G$ and extends by density using \eqref{eq:K42} and Proposition \ref{prop:kishimoto}). By Corollary \ref{cor:Lext}, $(\pd_t+\pd_x^3)\Lc h(u)=0$ in $\mathcal D'((0,\infty)\times\R)$. Hence $\pd_tu+\pd_x^3u+\pd_x(u^2)=0$ in $\mathcal D'((0,\infty)\times(0,1))$. (c) $u\in C([-1,1];H^{-3/4}(\R))$ since each term is (Lemma \ref{lem:kishimoto-linear}, Corollary \ref{cor:Lext}), and $u(\cdot,0)=\tilde u_0+0+\Lc h(u)(\cdot,0)=\tilde u_0$, whose restriction to $\R^+$ is $u_0$. (d) Each term of \eqref{eq:Lambda} belongs to $C(\R_x;H^{1/12}(\R_t))$ (Lemma \ref{lem:KPV}, Lemma \ref{lem:trace}, Corollary \ref{cor:Lext}), and at $x=0$,
$$u(0,t)=\psi e^{-t\pd_x^3}\tilde u_0|_{x=0}-\psi\Kc\pd_x(u^2)|_{x=0}+h(u)(t)=\tilde f(t)\qquad\text{for }0<t<1,$$
because $\chi_{(0,\infty)}\psi_2=1$ on $(0,1)$; thus $u(0,\cdot)=f$ in $H^{1/12}(0,1)$.

Finally, the solution of Theorem \ref{thm:main} is the restriction of the fixed point $u$ to $\R^{+}\times[0,1]$.

Finally, the result for larger data follows by scaling, exactly as in \cite{Holmer}. This completes the proof of Theorem 1.2.


\section*{Acknowledgments}
M. Cavalcante acknowledges the kind hospitality of the Université de Tours during a sabbatical stay funded by the CAPES/COFECUB Program (grant no.~88887.879175/2023-00), during which part of this work was carried out, and thanks Luc Molinet for fruitful discussions on various topics during this visit. Finally, he acknowledges the support of CNPq, grant no.~306348/2025-0.

\end{document}